\documentclass[12pt]{article}
\usepackage{}
\usepackage{mathrsfs,graphicx,amsmath,amssymb,amsfonts}

\usepackage{appendix}
\usepackage{caption}
\usepackage{mathrsfs,graphicx,amsmath,amssymb,amsfonts,amsthm}
\usepackage{cite}
\usepackage{longtable}
\usepackage{supertabular}
\evensidemargin=\oddsidemargin\topmargin=-1.5cm

\newtheorem{lem}{Lemma}[section]

\newtheorem{thm}[lem]{Theorem}

\theoremstyle{definition}

\usepackage{tikz,xcolor,hyperref}
\definecolor{lime}{HTML}{A6CE39}
\DeclareRobustCommand{\orcidicon}{
\begin{tikzpicture}
\draw[lime, fill=lime] (0,0)
circle[radius=0.16]
node[white]{{\fontfamily{qag}\selectfont \tiny \.{I}D}};
\end{tikzpicture}
\hspace{-2mm}}
\foreach \x in {A, ..., Z}{%
\expandafter\xdef\csname orcid\x\endcsname{\noexpand\href{https://orcid.org/\csname orcidauthor\x\endcsname}{\noexpand\orcidicon}}}

\begin{document}
\title{The least distance eigenvalue of the complements of
graphs of diameter greater than three \thanks{This work is supported by Natural Science Foundation for Young People of Xinjiang Province (No. 2022D01B136).}}
\author{Xu Chen$^1$\hspace{-1.5mm}\orcidA{},
Yinfen Zhu$^2$,
Guoping Wang$^3$\footnote{Corresponding author. Email: xj.wgp@163.com.(G. Wang)}\\
{\small 1. School of Statistics and Data Science, Xinjiang University of Finance and Economics,
}\\
{\small {\"U}r{\"u}mqi, Xinjiang 830012, P.R.China;}\\
{\small 2. School of Mathematics and Science, Xinjiang Institute of Engineering, }\\
{\small {\"U}r{\"u}mqi, Xinjiang 830023, P.R.China;}\\
{\small 3. School of Mathematical Sciences, Xinjiang Normal University, }\\
{\small {\"U}r{\"u}mqi, Xinjiang 830017, P.R.China}}

\date{}
\maketitle {\bf Abstract.}
Suppose $G$ is a connected simple graph with the vertex set
$V( G ) = \{ v_1,v_2,\cdots ,v_n  \} $.
Let $d_G( v_i,v_j ) $ be the least distance between $v_i$ and $v_j$ in $G$.
Then the distance matrix of $G$ is $D( G ) =( d_{ij} ) _{n\times n}$,
where $d_{ij}=d_G( v_i,v_j ) $.
Since $D( G )$ is a non-negative real symmetric matrix,
its eigenvalues can be arranged as
$\lambda_1(G)\ge \lambda_2(G)\ge \cdots \ge \lambda_n(G)$,
where eigenvalue $\lambda_n(G)$
is called the least distance eigenvalue of $G$.
In this paper we determine the unique graph whose least distance eigenvalue attains maximum among all complements of graphs of diameter greater than three.

{\flushleft{\bf MR(2020) Subject Classification:}}  05C12, 05C50\\

\section{Introduction}

~~~~~The complement of graph $G=( V( G ) ,E( G ) )$ is denoted by $G^c=( V( G^c ) ,E( G^c ) )$,
where $V( G^c ) =V( G ) $ and $E( G^c ) = \{ xy\notin E(G):x,y\in V( G )\} $.
The spectrum and distance spectrum of complements of graphs have been studied, see references \cite{J.G.S, Y.G.D, L.H.Q.2, Q.R, L.S.C, chenxu,  F.Y.Z}.
X. Chen and G. Wang \cite{chenxu} determined the unique graphs whose distance spectral radius respectively attains maximum and minimum among all complements of graphs of diameter greater than three,
and the unique graph whose least distance eigenvalue attains minimum
among all complements of graphs of diameter greater than three.
In this paper, we determine the unique graph whose least distance eigenvalue attains maximum among all complements of graphs of diameter greater than three.

\section {Main results}

~~~~Let $G$ be a connected simple graph with the vertex set $V( G ) = \{ v_1,v_2,\cdots ,v_n  \} $.
Then the adjacenct matrix of $G$ is $A( G ) =( a_{ij} ) _{n\times n}$,
where $a_{ij}=1$ if $v_i$ is adjacent to $v_j$, and $a_{ij}=0$ otherwise.
Let $J_n$ be the matrix of order $n$ whose all entries are $1$, and $I_n$ be the identity matrix of order $n$.
Suppose $A=(a_{ij})_{n\times n}$ and $B=(b_{ij})_{n\times n}$.
Then we write $A=B$ if $a_{ij}=b_{ij}$, and $A\ge B$ if $a_{ij}\ge b_{ij}$.
We denote by $d(G)$ the diameter of $G$, which is the farthest distance between all pairs of vertices.

The below Lemma 2.1 reflects the relationship of $D(G^c)$ and $A(G)$.

\begin{lem}
[\cite{chenxu}]
Suppose $G$ is a simple connected graph on $n$ vertices whose diameter $d(G)$ is greater than two.
Then we have
\begin{enumerate}
\setlength{\parskip}{0ex}
\item[\rm (I.)]
when $d( G ) > 3$, $D( G^c ) =J_n-I_n+A( G )$.
\item[\rm (II.)]
when $d( G ) =3$, $D( G^c )\geq J_n-I_n+A( G )$.
\end{enumerate}
\end{lem}

In this paper we always assume that the diameter of $G$ is greater than three, and so
$G$ and its complement $G^c$ are both connected.
Using the relations between $D(G^c)$ and $A(G)$ stated in Lemma 2.1
we determine the unique graph whose least distance eigenvalue attains maximum
among all complements of graphs.
\vskip 3mm

If two vertices $u$ and $v$ are adjacent
then we write $u\thicksim v$.
Suppose $G$ is a connected simple graph with the vertex set $V(G) =\{ v_1,v_2,\cdots ,v_n \}$.
Let $x=( x_1,x_2,\cdots ,x_n ) ^T$ be an eigenvector of $A(G)$,
where $x_i$ corresponds to $v_i$, i.e., $x( v_i ) =x_i$ for $i=1,2,\cdots ,n$. Then
\begin{equation}
x^TA(G)x=2\sum_{v_i\thicksim v_j}{x_ix_j}
\end{equation}

Suppose $y=(y_1,y_2,\cdots ,y_n)^T$ is an eigenvector of $D(G)$ with respect to the eigenvalue $\lambda$,
where $y_i$ corresponds to $v_i$, i.e., $y( v_i ) =y_i$ for $i=1,2,\cdots ,n$. Then
\begin{equation}
\lambda y_i=\sum_{v_j\in V( G )}{d_{ij}y_j}.
\end{equation}


\begin{lem}
Suppose $G$ is connected graph of diameter greater than three on $n$ vertices.
Let $G_{T_1}$ be the connected graph of diameter greater than three obtained from $G$ by deleting an edge $uv$ and connecting $u$ and $v_1$.
Set $x$ to be an eigenvector of $D(G^c)$ with respect to $\lambda_n(G^c)$.
If $x(u)x(v)\ge x(u)x(v')$,
then $\lambda_n(G^c)\ge \lambda _n(G_{T_1}^c)$.
\end{lem}

\begin{proof}
Recall that $x(u)x(v)\ge x(u)x(v')$.
Then we easily have $x^TA( G ) x\ge x^TA(G_{T_1}) x$.
Note that $d( G ) >3$ and $d(G_{T_1}) >3$.
By Lemma 2.1 and equation (1) we have
\begin{equation}
\begin{split}
\lambda_n(G^c) &=x^TD(G^c)x\\
&=x^T( J_n-I_n ) x+x^TA( G ) x\\
&\ge x^T( J_n-I_n ) x+x^TA(G_{T_1}) x\\
&\ge x^TD(G_{T_1}^c) x.
\nonumber
\end{split}
\end{equation}
By Rayleigh's theorem we have $x^TD(G_{T_1}^c) x \ge \lambda _n(G_{T_1}^c)$,
and so $\lambda_n(G^c)\ge \lambda _n(G_{T_1}^c)$.
\end{proof}

Let $x=(x_1,x_2,\cdots ,x_n)$ be an eigenvector of $D(G^c)$  with respect to $\lambda_n(G^c)$, where $x(v_i)=x_i$ $(x(v_i)=x_i$ $(i=1,2,\cdots,n))$.
Write $V_+=\{v_i\in V(G^c): x_i>0\}$, $V_-=\{v_i\in V(G^c): x_i<0\}$ and $V_0=\{v_i\in V(G^c): x_i=0\}$.
Let $S$ be a subset of $V(G)$. We denote by $G[S]$ the subgraph of $G$ induced by $S$.

\begin{lem}
Suppose $G$ is a connected graph of diameter greater than three.
Let $x$ be an eigenvector of $D(G^c)$  with respect to $\lambda_n(G^c)$.
Let $G_1$ be a connected graph obtained from $G$ by deleting one edge in
either $G[V_+\cup V_0]$ or $G[V_-]$.
Then $\lambda_n(G^c)\ge \lambda _n(G_1^c)$.
\end{lem}

\begin{proof} We first have $x^TA(G) x\ge x^TA(G_1) x$.
Note that $d(G_1)\ge d(G)>3$.
By Lemma 2.1 and the equation (1) we have
\begin{equation}
\begin{split}
\lambda_n(G^c) &=x^TD(G^c)x\\
&=x^T( J_n-I_n ) x+x^TA( G ) x\\
&\ge x^T( J_n-I_n ) x+x^TA(G_1) x\\
&\ge x^TD(G_1^c) x.
\nonumber
\end{split}
\end{equation}
By Rayleigh's theorem we have $x^TD(G_1^c) x \ge \lambda _n(G_1^c)$,
and so $\lambda_n(G^c)\ge \lambda _n(G_1^c)$.
\end{proof}

\begin{lem}
Suppose $G$ is a connected graph of diameter greater than three.
Set $x$ to be an eigenvector of $D(G^c)$ with respect to $\lambda_n(G^c)$.
Let $G_2$ be a connected graph of diameter greater than three obtained from $G$ by connecting one pairs of vertices between $V_+\cup V_0$ and $V_-$ which are not adjacent.
Then $\lambda_n(G^c)\ge \lambda _n(G_2^c)$.
\end{lem}

\begin{proof}
We first have $x^TA( G ) x\ge x^TA(G_2) x$.	
Note that $d(G)>3$ and $d(G_2)>3$.
By Lemma 2.1 and the equation (1) we have
\begin{equation}
\begin{split}
\lambda_n(G^c) &=x^TD(G^c)x\\
&=x^T( J_n-I_n ) x+x^TA( G ) x\\
&\ge x^T( J_n-I_n ) x+x^TA(G_2) x\\
&\ge x^TD(G_2^c) x.
\nonumber
\end{split}
\end{equation}
By Rayleigh's theorem we have $x^TD(G_2^c) x \ge \lambda _n(G_2^c)$,
and so $\lambda_n(G^c)\ge \lambda _n(G_2^c)$.
\end{proof}

Let $x$ be an eigenvector of $D(G^c)$  with respect to $\lambda_n(G^c)$.
If there are $s$ edges of $G[V_+\cup V_0]$ and $G[V_-]$ and $t$ edges between $V_+\cup V_0$ and $V_-$,
then we write $G$ for $G_{s,t}$.
Let $\mathbb{G}_{s+1,t}$ be a set of such graphs of diameter greater than three obtained from $G_{s,t}$ by adding an edge in $G[V_+\cup V_0]$ and $G[V_-]$, and $\mathbb{G}_{s,t-1}$ be a set of such graphs of diameter greater than three obtained from $G_{s,t}$ by deleting an edge between $V_+\cup V_0$ and $V_-$.
By applying Lemmas 2.2-2.4 we easily obtain the following lemma.

\begin{lem}
Let $G_{s,t}$ be defined as above.
If $d(G_{s,t})>3$ then we have
\begin{equation}
\begin{split}
\lambda_n(G_{s,t}^c)\le \max
\left\{\mathop{\max}_{G\in \mathbb{G}_{s+1,t}} \lambda_n(G^c), \mathop{\max}_{G\in \mathbb{G}_{s,t-1}} \lambda_n(G^c)\right\}.
\nonumber
\end{split}
\end{equation}
\end{lem}

Let $K_{n-2}$ be a complete graph of order $n-2$.
We denote by $K'$ the graph by deleting an edge $uv$ of $K_{n-2}$ and respectively appending vertices $u'$ and $v'$ to $u$ and $v$.
Let $K''$ denote the graph by deleting an edge $uv$ of $K_{n-2}$ and appending a path of order $2$ to the vertex $u$.
Clearly, $d(K')=d(K'')=4$.

\begin{lem}
Let $K'$ and $K''$ be defined as above.
Then we have $\lambda_n(K''^c)\le \lambda_n(K'^c)$.
\end{lem}

\begin{proof}
Set $x$ to be the eigenvector of $D(K'^c)$ with respect to $\lambda_n(K'^c)$.
Write $\lambda_n(K'^c)$ for $\lambda_n$.
By the symmetry of $K'^c$ all the vertices in $(V_+\cup V_0)\setminus \{u,v,u',v'\}$ correspond to the same value $x_1$.
Let $x(u)=x_u$, $x(v)=x_v$, $x(u')=x_{u'}$
and $x(v')=x_{v'}$.
By the equation (2) we have
$$ \left\{ \begin{array}{l}
	\lambda_nx_u=2x_{u'}+x_v+x_{v'}+2(n-4)x_1,\\
	\lambda_nx_{u'}=2x_u+x_v+x_{v'}+(n-4)x_1,\\
	\lambda_nx_v=x_u+x_{u'}+2x_{v'}+2(n-4)x_1,\\
	\lambda_nx_{v'}=x_u+x_{u'}+2x_v+(n-4)x_1,\\
	\lambda_nx_1=2x_u+x_{u'}+2x_v+x_{v'}+2(n-5)x_1.
\end{array}\right.$$

We can transform the above equation into a matrix equation $\lambda_nI_5-D_{K'^c}x'=0$,
where $x'=(x_u,x_{u'},x_v,x_{v'},x_1)^T$ and
$$D_{K'^c}=\left( \begin{matrix}
	0&2&1&1&2(n-4)\\
	2&0&1&1&n-4\\
	1&1&0&2&2(n-4)\\
	1&1&2&0&n-4\\
	2&1&2&1&2(n-5)
\end{matrix} \right).$$

Let $\phi(\lambda)=\det (I_5\lambda-D_{K'^c})$.
Then we have
\begin{equation}
	\phi(\lambda)=\lambda^5-(2n-10)\lambda^4-(10n-28)\lambda^3-10n\lambda^2+(4n-48)\lambda.
\end{equation}

Similarly, we have
\begin{equation}
	\begin{split}
		\varphi(\lambda)=&\det(I_5\lambda-D_{K''^c})\\
		=&\lambda^5-(2n-10)\lambda^4-(10n-28)\lambda^3\\
		&-(8n+10)\lambda^2+(15n-103)\lambda+14n-70.
	\end{split}
\end{equation}

From the equations (3) and (4), we have
$$\phi(\lambda)-\varphi(\lambda)=(-2n+10)\lambda^2+(-11n+55)\lambda-14n+70.$$

Let $T$ be the tree obtained by appending two pendent vertices to the some one end of the path $P_3$ of order $3$.

By the complements of $K'$ and $K''$, we observe that the tree $T$ of order $5$ is an induced subgraph of $K'^c$ and $K''^c$, and so both
$D(K'^c)$ and $D(K''^c)$ contain a principle submatrix $D(T)$.
Whereas $\lambda_5(T)<-3.8$,
by Interlacing theorem we have $\lambda_n(K'^c)<-3.8$ and $\lambda_n(K''^c)<-3.8$.
Therefore, we can compute out that $\phi(\lambda)-\varphi(\lambda)<0$ if $\lambda<-3.8$ and $n\ge 7$.
This implies $\lambda_n(K''^c)\le \lambda_n(K'^c)$.
\end{proof}

The neighbor $N_G(v)$ of the vertex $v$
of $G$ is the set of the vertices which are
adjacent to $v$.
Suppose $G$ is a simple graph of diameter than three with the vertex set $V(G)=\{v_1,v_2,\cdots v_n\} (n\ge 7)$.
Set $x$ to be an eigenvector of $D(G^c)$ with resepct to $\lambda_n(G^c)$.
Let $|V_+\cup V_0|=p$ and $|V_-|=q$.
Without loss of generality in what follows we assume that $p\ge q$.

\begin{lem}
Suppose $G$ is a simple graph of diameter greater than three on $n\ge 7$ vertices.
If $q=1$ then $\lambda_n(G^c)\le \lambda_n(K'^c)$.
\end{lem}

\begin{proof}
Since $d(G)>3$, there must be the shortest path $P=u_1u_2u_3u_4u_5$.
Let $V_-=\{v\}$.
Now we distinguish four cases as follows.

{\bf Case 1.} Suppose $v=u_3$.

Without loss of generality we assume that $x(u_1)\ge x(u_5)$.
We denote by $G'$ the graph obtained from $G$ by deleting all edges which are incident to $v$ except $vu_2$,
and deleting all edges between $u_5$ and the vertices in $N_G(u_5)\setminus {u_4}$ and adding edges between $u_1$ and the vertices in $N_G(u_5)\setminus{u_4}$.
Connecting all pairs of vertices of $G'\setminus \{v,u_5\}$ which are not adjacent except $u_2$ and $u_4$ in $G'$, we get a graph isomorphic to $K'$.
It follows from Lemma 2.5 that $\lambda_n(G^c)\le \lambda_n(K'^c)$.

{\bf Case 2.} Suppose $v=u_2$ (or $v=u_4$).

If $x(u_1)\ge x(u_5)$, we denote by $\widetilde{G}$ the graph obtained from $G$ by deleting all edges which are incident to $v$ except $vu_1$,
and deleting all edges between $u_5$ and the vertices in $N_G(u_5)\setminus{u_4}$ and adding edges between $u_1$ and the vertices in $N_G(u_5)\setminus{u_4}$.
Connecting all pairs of vertices of $V(\widetilde{G})\setminus \{v,u_5\}$ which are not adjacent except $u_1$ and $u_4$ in $\widetilde{G}$,
we get a graph isomorphic to $K'$.
It follows from Lemma 2.5 that $\lambda_n(G^c)\le \lambda_n(K'^c)$.

So we assume that $x(u_1)<x(u_5)$.
We denote by $\widetilde{G}'$ the graph obtained from $G$ by deleting all edges which are incident to $v$ except $vu_3$,
and deleting all edges between $u_1$ and the vertices in $N_G(u_1)$ and adding edges between $u_5$ and $N_G(u_1)$.
Connecting all pairs of vertices of $V(\widetilde{G}'\setminus \{v,u_1\})$
which are not adjacent except $u_3$ and $u_5$,
and connecting $u_1$ and $u_5$ in $\widetilde{G}'$,
we get a graph isomorphic to $K'$.
It follows from Lemma 2.5 that $\lambda_n(G^c)\le \lambda_n(K'^c)$.

{\bf Case 3.} Suppose $v=u_1$ (or $v=u_5$).

If $x(u_2)\ge x(u_5)$, we denote by $\bar{G}$ the graph obtained from $G$ by deleting all edges which are incident  to $u_5$ except $u_5u_4$.
Connecting all pairs of vertices of $V(\bar{G})\setminus \{v,u_5\}$ which are not adjacent except $u_2$ and $u_4$ in $\bar{G}$,
we get a graph isomorphic to $K'$.
It follows from Lemma 2.5 that $\lambda_n(G^c)\le \lambda_n(K'^c)$.

So we assume that $x(u_2)<x(u_5)$.
We denote by $\bar{G}'$ obtained from $G$ by deleting all edges which are incident to $v$ except $vu_2$,
and deleting all edges between $u_2$ and the vertices in $N_G(u_2)\setminus{\{u_3,v\}}$ and adding edges between $u_5$ and the vertices in $N_G(u_2)\setminus{\{u_3,v\}}$.
Connecting all pairs of vertices of $V(\bar{G}')\setminus \{v,u_2\}$ which are not adjacent  except $u_3$ and $u_5$ in $\bar{G}'$, we get a graph isomorphic to $K''$.
It follows from Lemma 2.5 that $\lambda_n(G^c)\le \lambda_n(K''^c)$.

{\bf Case 4.} Suppose $v$ is adjacent to $u_3$.

Without loss of generality assume that $x(u_1)\ge x(u_5)$.
We denote by $\hat{G}$ the graph obtained from $G$ by deleting all edges which are incident to $v$ except $vu_3$,
and deleting all edges between $u_5$ and the vertices in $N_G(u_5)$ and adding edges between $u_1$ and the vertices in $N_G(u_5)$
and connecting $u_5$ and $u_1$.
Connecting all pairs of vertices of $V(\hat{G})\setminus \{v,u_5\}$ except $u_1$ and $u_3$ in $\hat{G}$,
we get a graph isomorphic to $K'$.
It follows from Lemma 2.5 that $\lambda_n(G^c)\le \lambda_n(K'^c)$.

If none of the above four cases occur
then by deleting some edges incident to $v$
and connecting some edges in $V(G)\backslash v$ we easily obtain a graph
isomorphic to $K'$ or $K''$.
It follows from Lemma 2.5 that $\lambda_n(G^c)\le \max \{\lambda_n(K'^c), \lambda_n(K''^c)\}.$

Thus, by Lemma 2.6 we finally have $\lambda_n(G^c)\le \lambda_n(K'^c)$.
\end{proof}

Let $a\ge 3$ and $b\ge 2$ be two integers and $a+b=n$.
Denote by $K(a,b)$ the graph by deleting an edge $uw$ of a completed graph $K_a$ and connecting $u$ and $v$ of another completed graph $K_b$.
Clearly, $d(K(a,b))=4$.

\begin{lem}
Suppose $G$ is a simple graph of diameter greater than three on $n$ ($\ge 7$) vertices.
If $q\ge 2$ then $\lambda_n(G^c)\le \lambda_n(K^c(p,q))$.
\end{lem}

\begin{proof}
Note that $G$ is a connected graph.
There must be one edge $uv$, where $u\in V_+\cup V_0$ and $v\in V_-$.
Since $d(G)>3$, there exists a vertex $w$ which is adjacent to neither $u$ nor $v$.
Without loss of generality we assume $w\in V_+\cup V_0$.
Connecting all pairs of vertices in $G[V_+\cup V_0]$ and $G[V_-]$ which are not adjacent except $u$ and $w$,
and deleting all edges between $V_+\cup V_0$ and $V_-$ except $uv$ in $G$,
we get a graph isomorphic to $K(p,q)$.
By Lemma 2.5 we have
$\lambda_n(G^c)\le \lambda_n(K^c(p,q))$.
\end{proof}

\begin{lem}
 $\lambda_n(K'^c)<\lambda_n(K^c(p,q))$.
\end{lem}

\begin{proof}
Let $x$ be the eigenvector of $D(K^c(p,q))$ with respect to $\lambda_n(K^c(p,q))$.
Write $\lambda_n(K^c(p,q))$ for $\lambda_n$.
By the symmetry of $K(p,q)$ all the vertices in $(V_+\cup V_0)\setminus \{u,w\}$ correspond to the same value $x_1$ and all the vertices in $V_-\setminus \{v\}$ correspond to the same value $x_2$.
Let $x(u)=x_u$, $x(v)=x_v$ and $x(w)=x_w$.
By the equation (2) we have
$$ \left\{ \begin{array}{l}
\lambda_nx_u=2x_v+x_w+2(p-2)x_1+(q-1)x_2,\\
\lambda_nx_v=2x_u+x_w+(p-2)x_1+2(q-1)x_2,\\
\lambda_nx_w=x_u+x_v+2(p-2)x_1+(q-1)x_2,\\
\lambda_nx_1=2x_u+x_v+2x_w+2(p-3)x_1+(q-1)x_2,\\
\lambda_nx_2=x_u+2x_v+x_w+(p-2)x_1+2(q-2)x_2.
\end{array}\right.$$

We can transform the above equation into a matrix equation $(\lambda_nI_5-D_{K^c(p,q)})x'=0$,
where $x'=(x_u,x_v,x_w,x_1,x_2)^T$ and
$$D_{K^c(p,q)}=\left( \begin{matrix}
0&2&1&2(p-2)&q-1\\
2&0&1&p-2&2(q-1)\\
1&1&0&2(p-2)&q-1\\
2&1&2&2(p-3)&q-1\\
1&2&1&p-2&2(q-2)
\end{matrix} \right).$$

Let $\psi_{p,q}(\lambda)=\det (I_5\lambda-D_{K^c(p,q)})$.
Then we have
\begin{equation}
\begin{split}
\psi_{p,q}(\lambda)=&{\lambda}^{5}- \left( 2\,q-10+2\,p \right) {\lambda}^{4}\\
&- \left( -3\,pq+16\,p+16\,q-40 \right) {\lambda}^{3}\\
&- \left( -18\,pq+44\,p+50\,q-74\right) {\lambda}^{2}\\
&- \left( -30\,pq+45\,p+63\,q-53 \right) \lambda+
12\,pq-10\,p-22\,q+2.
\end{split}
\end{equation}

From the equations (3) and (5) we have
$$\psi_{p,q}(\lambda)-\phi(\lambda)= \left( 2\,p-6 \right) {\lambda}^{2}+ \left( 11\,p-33 \right) \lambda+
14\,p-42.$$

Clearly the tree $T$ of order $5$ is an induced subgraph of $K^c(p,q)$, $D(K^c(p,q))$ contains $D(T)$ as a principle submatrix.
Whereas $\lambda_4(T)<-3.8$,
by Interlacing theorem we have $\lambda_n(K^c(p,q))<-3.8$.
Recall that $p\ge 4$ and $q\ge 2$.
Therefore, we can compute out that $\psi_{p,q}(\lambda)-\phi(\lambda)>0$ if $\lambda<-3.8$.
By Lemma 2.6 we know $\lambda_n(K'^c)<-3.8$, and so
$\lambda_n(K'^c)<\lambda_n(K^c(p,q))$.
\end{proof}

\begin{lem}
$\lambda_n(K^c(p,q))<\lambda_n(K^c(\left\lceil {\frac{n}{2}} \right\rceil ,\left\lfloor {\frac{n}{2}} \right\rfloor))$.
\end{lem}

\begin{proof}
By the equation (5) we have
\begin{equation}
\begin{split}
\psi_{p,q}(\lambda)-\psi_{p-1,q-1}(\lambda)=& \left( -3\,p+3\,q+3 \right) {\lambda}^{3}+ \left( -18\,p+18\,q+24
\right) {\lambda}^{2}\\
&+ \left( -30\,p+30\,q+48 \right) \lambda-12\,p+
12\,q+24.
\nonumber
\end{split}
\end{equation}
Recall that $p>q\ge 2$.
By computation we obtain that $\psi_{p,q}(\lambda)-\psi_{p-1,q-1}(\lambda)>0$ if $\lambda<-3$.
Note that $n=p+q$.
By Lemma 2.7 we know $\lambda_n(K^c(p,q))<-3.8$, and so $\lambda_n(K^c(p,q))<\lambda_n(K^c(\left\lceil {\frac{n}{2}} \right\rceil ,\left\lfloor {\frac{n}{2}} \right\rfloor))$.
\end{proof}

Combining Lemmas 2.8-2.10 we have the following theorem.

\begin{thm}
Suppose $G$ is a simple graph of diameter greater than three on $n\ge 7$ vertices. Then
$$\lambda_n(G^c)<\lambda_n\left(K^c\left(\left\lceil {\frac{n}{2}} \right\rceil ,\left\lfloor {\frac{n}{2}} \right\rfloor\right)\right).$$
\end{thm}

\end{document}